\documentclass[12pt, reqno]{amsart}
\usepackage{hyperref}

\usepackage{enumerate}
\usepackage{amssymb,amsmath,graphicx,amsthm, fullpage}
\usepackage{xcolor}
\usepackage{soul}

\usepackage{color}

\usepackage[utf8]{inputenc}
\usepackage[T1]{fontenc}
\usepackage[english]{babel}
\usepackage{amsmath, amssymb, amsthm, amsfonts}
\usepackage{geometry}
\usepackage{hyperref}
\hypersetup{colorlinks=true, linkcolor=blue, citecolor=red}

\newtheorem{theorem}{Theorem}[section]
\newtheorem{lemma}[theorem]{Lemma}

\newtheorem{corollary}[theorem]{Corollary}

\theoremstyle{definition}

\newtheorem{remark}[theorem]{Remark}

\newcommand{\C}{\mathbb{C}}
\newcommand{\D}{\mathbb{D}}
\newcommand{\R}{\mathbb{R}}
\newcommand{\Z}{\mathbb{Z}}
\newcommand{\dbar}{\bar{\partial}}
\newcommand{\di}{\partial}

\newcommand{\wbar}{\bar{w}}
\newcommand{\zbar}{\bar{z}}

\newcommand{\w}{\bar{w}}
\newcommand{\z}{\bar{z}}
\def\H{\mathbb H}
\newcommand{\Lnorm}[2]{L_{\vphantom{x}#1}^{\vphantom{p'}#2}}

\def\z{{\bar z}}
\def\w{{\bar w}}
\def\di{\partial}

\def\bw{\bar w}

\begin{document}

\title[$L^p$ Estimates for $\dbar$ on Rational Hartogs Triangles]{$L^p$ Estimates for the $\dbar$-Problem on Rational Hartogs Triangles}

\author{Tran Vu Khanh}
\address{Department of Mathematics, International University, Vietnam National University Ho Chi Minh City, Vietnam}
\email{tvkhanh@hcmiu.edu.vn}

\author{Tu Nguyen$^*$} \thanks{$^*$Corresponding author.}
\address{Department of Mathematics, International University, Vietnam National University Ho Chi Minh City, Vietnam}
\email{natu@hcmiu.edu.vn}

\thanks{This research is funded by Vietnam National Foundation for Science and Technology
	Development (NAFOSTED) under grant number IZVSZ2 229554.}


\maketitle
\begin{abstract}
We investigate $L^p$ estimates for the $\bar{\partial}$-problem on rational Hartogs triangles $\H_{m/n} = \{ (z_1, z_2) \in \mathbb{C}^2 : |z_1|^m < |z_2|^n < 1 \}$. For  $p \in (1, \infty)$, we establish the existence of a solution operator that is bounded on $L^p(\H_{m/n})$. Our approach avoid the need for any {\it a priori} condition on the data.
	
We also show that the canonical solution $K_{\H_{m/n}}$ is bounded on $L^p(\H_{m/n})$ for $p \in (p_0, p_2)$, where $p_0=\frac{2m+2n}{m+n+1+\min\{m, n\}}$ and $p_2=\frac{2m+2n}{m+n-1}$. For classical Hartogs triangle, $\H_1$, this establishes boundedness for $p \in (1, 4)$. 
\end{abstract}

\section{Introduction}

For $\gamma > 0$, the Hartogs triangle of exponent $\gamma$ is defined as the pseudoconvex domain
\begin{equation} \label{eq:hartogs_def}
	\H_{\gamma} = \left\{ (z_1, z_2) \in \mathbb{C}^2 : |z_1|^\gamma < |z_2| < 1 \right\}.
\end{equation}
We distinguish between \textit{rational} Hartogs triangles, where $\gamma = m/n \in \mathbb{Q}^+$ with $\gcd(m,n)=1$, and \textit{irrational} ones. 

The classical Hartogs triangle, $\H_1$, has long served as an important source of counterexamples in several complex variables. The primary challenge in the study of various operators on $\H_{\gamma} $ is that its boundary fails to be a graph at the origin (see \cite{Sh15, Sh23, Sh24, BFLS22, MaMi92, ChCh91, Sib75}).

The mapping properties of the Bergman projection $B_{\H_\gamma}$ have been a subject of intense investigation in recent years. Unlike for smoothly bounded or Lipschitz pseudoconvex domains, where $L^p$ boundedness of Bergman projection typically holds for the full range $1 < p < \infty$, for  Hartogs triangles it only holds for a restricted range of $p$ depending on $\gamma$. As established in \cite{EdMc16, EdMc17, ChZe16, ChEdMc19}, $B_{\H_{m/n}}$ is bounded on $L^p(\H_{m/n})$ if and only if $p \in (p_1, p_2)$, where 
\begin{equation} \label{eq:p1p2_def}
	p_1 = \frac{2m+2n}{m+n+1} \quad \text{and} \quad p_2 = \frac{2m+2n}{m+n-1}.
\end{equation}
That weak-type estimate holds at $p_2$ but not at $p_1$ was showed in \cite{HuWi20, ChAdKo23}. For irrational $\gamma$,  $B_{\H_{\gamma}}$ is bounded only for the trivial case $p=2$ \cite{Edh16}. 
Similar phenomena have been observed for weighted Bergman projections, Toeplitz operators, and the Berezin transform (see \cite{BTRWZ22, KhLiTh17, We25}).

The $L^p$ theory for the Cauchy--Riemann equation $\bar{\partial} u = f$ on $\H_1$ is similarly delicate. We briefly summarize some important results. 
\begin{itemize}
	\item Chaumat and Chollet \cite{ChCh91} and Ma and Michel \cite{MaMi92} established that the $\bar{\partial}$-equation admits bounded solution operators in H\"older spaces $C^{k,\alpha}(\bar{\H}_1)$ for any $k \in \mathbb{N}$ and $\alpha \in (0,1)$. They also contruct $\bar{\partial}$-closed forms in $C^\infty(\bar{\H}_1)$ that admit no solution smooth up to the boundary, a phenomenon rooted in the fact that $\bar{\H}_1$ lacks a basis of Stein neighborhoods \cite{Sib75}.
	
	\item Chakrabarti and Shaw \cite{ChSh13} addressed the failure of global regularity by establishing weighted Sobolev estimates, with weights of the form $|z_2|^{2\ell}$.
	
	\item In the $L^p$ setting, Chen and McNeal \cite{ChMc20} constructed an explicit integral solution operator on $\H_1$ and proved its boundedness on $L^p(\H)$ for $p\in(1,\infty)$. However, they required that the data $f = f_1d\zbar_1+ f_2d\zbar_2  $ satisfies $\zbar_1f_1+\zbar_2 f_2 =0$. (This is equivalent to requiring $g_2=0$ in \eqref{pullback_components}.)

	\item Recently, Zhang \cite{Zha24} constructed a bounded solution operator without any extra condition on $f$, but only for $p \in [4, \infty)$. This result was further extended to Sobolev spaces $W^{k,p}$ by Pan and Zhang \cite{PaZh25} for the same range of $p$. The restriction $p \ge 4$ appears to be a limit of their approach, rather than an intrinsic barrier to the problem.
\end{itemize}

The first objective of this paper is to overcome the restriction $p\ge 4$. We prove $L^p$ estimates for the $\bar{\partial}$-problem on all rational Hartogs triangles and for the {\it full range} $1 < p < \infty$, without imposing any \textit{a priori} vanishing conditions on the data.

\begin{theorem} \label{thm:main_sol_op}
	 For any $1 < p < \infty$, there exists a solution operator $T_p $ for the $\bar{\partial}$-problem on $\H_{m/n}$ such that 
	\[ T_p : L^p_{(0,1)}(\H_{m/n}) \to L^p(\H_{m/n}) \]
	is bounded.
\end{theorem}
In fact, we will prove a weighted estimate that contains this result as a special case, see Theorem \ref{thm:main_weighted}. Our approach employs a branched covering map $\psi: \D \times \D^* \to \H_{m/n}$ to reduce the problem to a weighted estimate for $\bar{\partial}$-problem on $\D \times \D^*$. After solving the $\bar{\partial}$-problem on $\D \times \D^* $, we apply an averaging process to transfer the solution back to  $\H_{m/n}$.

In the second part of this paper, we investigate the $L^p$ mapping properties of the \textit{canonical solution} to the $\dbar$ problem on $\H_{m/n}$. This is the $L^2$-minimal solution, which by  Kohn's formula is given by
\begin{equation} \label{eq:kohn_rep}
	K_{\H_{m/n}} = (I - B_{\H_{m/n}}) T_2 : L_{(0,1)}^2(\H_{m/n}) \cap \ker(\bar{\partial}) \to L^2(\H_{m/n}).
\end{equation}
For $p<2$, as long as it is defined, $(I - B_{\H_{m/n}}) T_p$ is an extension of $K_{\H_{m/n}} $, since $I - B_{\H_{m/n}}$ annihilates  $A^2(\H_{m/n})$. We will continue to denote this extension $K_{\H_{m/n}} $.
For $p \in (p_1, p_2)$,  the boundedness of $K_{\H_{m/n}}$ on $L^p$ follows directly from Theorem \ref{thm:main_sol_op} and the result on the Bergman projection mentioned above.

The second contribution of this work is to extend this range. By combining our weighted estimates and the results for Bergman--Toeplitz operators from \cite{BTRWZ22}, we obtain the following
\begin{theorem} \label{thm:canonical_rational} Let $K_{\H_{m/n}}$ be the canonical solution operator on the Hartogs triangle $\H_{m/n}$ and
	\begin{equation} \label{p0}
		p_0 = \frac{2m+2n}{m+n+1+\min\{m, n\}}.
\end{equation}
Then 
\begin{itemize}
	\item[(i)] $K_{\H_{m/n}}$ is bounded on $L^p(\H_{m/n})$ if $p \in (p_0, p_2)$,
	\item[(ii)] $K_{\H_{m/n}}$ is not bounded on $L^p(\H_{m/n})$ if $p \ge p_2$.
\end{itemize}
\end{theorem}
For the classical Hartogs triangle $\H_1$, we have $p_0 = 1, p_1 = 4/3$, and $p_2 = 4$. This gives
\begin{corollary}
	The canonical solution $K_{\H_1}$ is bounded on $L^p(\H_1)$ for  $p \in (1, 4)$ and unbounded for $p\ge 4$.
\end{corollary}

The remainder of the paper is organized as follows. Section 2 is devoted to the $\bar{\partial}$-problem and weighted $L^p$ estimates on the unit disk $\D$ and the bidisc $\D \times \D$. In Section 3, we prove our key result Theorem \ref{thm:main_weighted} which implies Theorem \ref{thm:main_sol_op} and is the key ingredient in the proof of Theorem \ref{thm:canonical_rational} in Section 4. Finally, the Appendix contains several auxiliary result that used in previous sections, which may be of independent interest.

Throughout this paper, we use $p'=\frac{p}{p-1}$ to denote the dual exponent of $p$. The notation $A \lesssim B$ means $A \le CB$ for some universal positive constant $C$, and $A \sim B$ means $A \le CB$ and $B \le CA$. For $\phi > 0 $ on $D\subset \C^n$,  the weighted space $L^p(D, \phi) $ is defined as
\[ L^p(D, \phi) = \left\{ u : \|u\|_{L^p(D, \phi)}^p=\int_D |u(z)|^p \phi(z) dV(z) < \infty \right\}. \]
A $(0,1)$-form $f = \sum_{j=1}^n f_j d\bar{z}_j$ is said to be in $L^p_{(0,1)}(D, \phi)$ if all its coefficients $f_j$ belong to $L^p(D, \phi)$.

	\section{Estimates for the $\dbar$-problem on $\D$ and $\D\times\D$}	
	In this section, we review the theory of the $\dbar$-problem on the unit disk $\D \subset \C$. Let $dA(w) = \frac{i}{2\pi} dw \wedge d\wbar$ denote the normalized area measure on $\D$. For $u \in C^1(\bar{\D})$, the Cauchy-Pompeiu formula states that
	\begin{equation}\label{Cauchy-formula1}
		u(z) = \frac{1}{2\pi i} \int_{b\D} \frac{u(w)}{w - z} dw - \int_{\D} \frac{\di_{\wbar}u(w)} {w - z}dA(w), \quad z \in \D.
	\end{equation}
	The boundary Cauchy transform $C_{b\D}: L^\infty(b\D) \to A(\D)$ and the solid Cauchy transform $C_{\D}: L^\infty(\D) \to C(\D)$ are defined respectively by 
	\[ C_{ b\D} u(z) = \frac{1}{2\pi i} \int_{b\D} \frac{u(w)}{w - z} dw \quad \text{and} \quad C_\D[f](z) =  -\int_{\D} \frac{f(w)}{w - z}dA(w). \]
	It is well-known that $C_\D$ is a solution operator for the inhomogeneous Cauchy-Riemann equation $\di_{\zbar} u = f$ on $\D$. Thus, \eqref{Cauchy-formula1} can be interpreted as a homotopy formula
	\begin{equation}\label{CD}
		C_{b\D} u + C_{\D}[\di_{\z}u] = u \quad \text{for } u \in  C^1(\bar{\D}).
	\end{equation}
	The Bergman projection $B_\D: L^2(\D) \to A^2(\D)$ is given by
	\[ B_\D u(z) = \int_{\D} \frac{u(w)}{(1 - z\wbar)^2} dA(w). \]
The canonical ($L^2$-minimal) solution operator is given by $$S_\D = (I - B_\D)C_\D,$$
which admits the explicit representation
\begin{align}
	S_\D[f](z)&= \int_{\D} \frac{1 - |w|^2}{(w - z)(1 - z\bar{w})} f(w) dA(w) \label{canonical-kernel}
\end{align}
for  $f\in L^2(\D)$.

 Since $(I - B_\D) C_{b\D} = 0$ on $C^1(\bar{\D})$, the homotopy formula \eqref{CD} yields 
\begin{equation}\label{BD}
	B_\D u + S_\D[\partial_{\bar{z}}u] = u \quad \text{ for } u\in C^1(\bar{\D}).
\end{equation}
The boundedness of the Bergman projection $B_{\D}$ and the canonical solution $S_{\D}$ in weighted $L^p$ spaces are as follows.
\begin{theorem}\label{S-weight} 
	Assume that $1< p< \infty$, $\alpha <2p-2$ and ${\beta}>-2$. Then,
	\begin{itemize}
		\item[(i)] $\|B_{\D}f\|_{L^p(\D, |z|^\beta)} \lesssim \|f\|_{L^p(\D, |z|^{\alpha})}$, 
		\item[(ii)] $\|S_{\D}f\|_{L^p(\D, |z|^\beta)} +\|C_{\D}f\|_{L^p(\D, |z|^\beta)}\lesssim \|f\|_{L^p(\D, |z|^{\alpha})}$ provided $\alpha \le \beta +p$.
	\end{itemize}
\end{theorem}

\begin{proof}
Note that $B_{\D}$,$S_{\D}$   and $C_{\D}$ are bounded by $S_{-2,0,0}$, $S_{-1,1,-1}$ and $S_{0,0,-1}$, respectively (see the Appendix for the definition of $S_{a,b,c}$). It is simple to check that with our hypothesis, all the conditions in Theorem \ref{keybound} are satisfied, thus the theorem follows.
\end{proof}

The next result for the bidisc $D = \D \times \D$ will play an essential role later on. For $j=1, 2$, $B_j$ and $C_j$ denote the Bergman projection and the Cauchy operator corresponding to the $j$-th factor.

\begin{theorem} \label{thm:canonical_bidisk}
	Let $D = \D \times \D$. For a $(0,1)$-form $f = f_1 d\bar{z}_1 + f_2 d\bar{z}_2$ on $D$, let 
	\begin{equation} \label{eq:rep_K}
		K_Df = S_1 f_1 + B_1 S_2 f_2.
	\end{equation}
	If $1 < p < \infty$ and $\beta \in (-2, p-2)$, the folllowing estimate holds
	\begin{equation} \label{eq:est_K}
		\|K_D f\|_{L^p(D,  |z_2|^{\beta})} \lesssim \|f_1\|_{L^p(D,  |z_2|^{\beta})} + \|f_2\|_{L^p(D,|z_2|^{\beta+p})}.
	\end{equation}
    Furthermore, if $\dbar f=0$ then $\dbar K_D f =f$.
\end{theorem}

\begin{proof}

	
	Since it is well-known that $S_{\D}$ is bounded on $L^p$, by a straightforward application of Fubini's theorem, we have
	\begin{align*}
		\|S_1 f_1\|_{L^p(D, |z_2|^{\beta})} &\lesssim   \|f_1\|_{L^p(D, |z_2|^{\beta})}.
	\end{align*}
	Similarly for the second term, we use the boundedness of  $B_{\D}$ on $L^p$, and the second part of Theorem \ref{S-weight}  to obtain
	\begin{align*}
		\|B_1 S_2 f_2\|_{L^p(D, |z_2|^{\beta})}^p &= \int_{\D_2} |z_2|^{\beta} \left( \int_{\D_1} |B_1 S_2 f_2|^p dV_1 \right) dV_2 \\
		&\lesssim \int_{\D_2} |z_2|^{\beta} \left( \int_{\D_1} |S_2 f_2|^p dV_1 \right) dV_2 \\
		&= \int_{\D_1}  \left( \int_{\D_2} |S_2 f_2|^p |z_2|^{\beta} dV_2 \right) dV_1 \lesssim \|f_2\|_{L^p(D, |z_2|^{\beta+p})}^p.
	\end{align*}
	Combining these two results yields the desired estimate \eqref{eq:est_K}.
	
	Now assume that $\dbar f=0$. Since $D$ is a star-shaped domain, smooth closed forms are dense. Thus, we can assume that $f$ is  smooth.   Using $\partial_{\bar{z}_j} S_j = I$, $\partial_{\bar{z}_j} B_j = 0$, and $S_j \partial_{\bar{z}_j} = I - B_j$, we have
	\begin{align*}
		{\partial}_{\bar{z}_1}K_D f &= {\partial}_{\bar{z}_1}S_1 f_1 + {\partial}_{\bar{z}_1} B_1 S_2 f_2 =  f_1, \\
		{\partial}_{\bar{z}_2} K_D f &= S_1 {\partial}_{\bar{z}_2} f_1+ B_1 {\partial}_{\bar{z}_2} S_2 f_2 = S_1 \partial_{\bar{z}_1} f_2 + B_1 f_2  = f_2.
	\end{align*}	
	This concludes the proof.
\end{proof}

That \eqref{eq:rep_K} also gives solution to the $\dbar$-problem on $\D\times\D^*$ follows from the following simple extension of Lemma 3.2 in Zhang \cite{Zha24}.

\begin{lemma}\label{extension}
	Let $1 \le p < \infty$ and  $f=f_1d\w_1+f_2d\w_2\in L^1_{(0,1)}(\D \times \D^*)$ be a $\dbar$-closed form with $f_1\in L^p(\D\times \D^*, |w_2|^{p-2})$. Then $f$ is $\dbar$-closed  on $D$.
\end{lemma}
\begin{proof}
	We need to show that $$ \int_{\D^2} f_1 \di_{\bw_2}\varphi -f_2 \di_{\bw_1}\varphi  = 0,$$
	for any $\varphi  \in C_c^\infty(\D^2)$. 
	Let $\eta_\varepsilon$  be a smooth function in the $w_2$ variable satisfying $ \mathbf{1}_{\D \setminus B_{2\varepsilon }} \le \eta_\varepsilon  \le \mathbf{1}_{\D \setminus B_{\varepsilon }}$  and $|\nabla \eta_\varepsilon (w_2)|\lesssim |w_2|^{-1}$ for  $w_2\in \D$. Since  $f$ is a closed $(0,1)$-form on $\D \times \D^*$,
	\begin{equation}\label{equa1}
		\int_{\D^2} f_1 \varphi \di_{\bw_2}\eta_\varepsilon   + \eta_\varepsilon (f_1 \di_{\bw_2} \varphi -f_2 \di_{\bw_1}\varphi) = \int_{\D^2} f_1 \di_{\bw_2}(\eta_\varepsilon  \varphi) -f_2 \di_{\bw_1}(\eta_\varepsilon  \varphi) =0.
	\end{equation}
	We have 
	\begin{align}\label{equa2}
		\begin{split}
			|\int_{\D^2} f_1 \varphi \di_{\bw_2}\eta_\varepsilon  | & \lesssim \|\varphi\|_{L^\infty(D)} \int_{\D\times (B_{2\varepsilon  }\setminus B_\varepsilon)} |f_1|  |w_2|^{-1}\\
			&\le \|\varphi\|_{L^\infty(D)} (\int_{\D\times (B_{2\varepsilon} \setminus B_\varepsilon )} |f_1|^p  |w_2|^{p-2})^{1/p} (\int_{\D\times (B_{2\varepsilon } \setminus B_\varepsilon)}  |w_2|^{-2})^{1/p'}\\
			&\lesssim \|\varphi\|_{L^\infty(D)}(\int_{\D\times (B_{2\varepsilon }\setminus B_\varepsilon )} |f_1|^p  |w_2|^{p-2})^{1/p}.
		\end{split}
	\end{align}
	Since $f_1\in L^p(\D\times \D^*, |w_2|^{p-2})$, this vanishes as $\varepsilon \to 0$. Therefore,
	it follows from \eqref{equa1} that
	$$ \int_{\D^2} f_1 \di_{\bw_2}\varphi -f_2 \di_{\bw_1}\varphi = \lim_{\varepsilon \to 0} \int_{\D^2}  \eta_\varepsilon (f_1 \di_{\bw_2} \varphi -f_2 \di_{\bw_1}\varphi) = 0,$$
	completing the proof.  	
\end{proof}
	

\section{Proof of Theorem \ref{thm:main_sol_op}}


The main result of this section is the following theorem, which establishes sharp weighted $L^p$ estimates for a solution to the $\bar{\partial}$-problem on $\H_{m/n}$. Theorem \ref{thm:main_sol_op} follows by taking $\alpha=\alpha'=0$.

\begin{theorem} \label{thm:main_weighted}
	Let $1 < p < \infty$ and $\alpha \in \mathbb{R}$. There is a solution operator $T_{p,\alpha}$ for the  $\bar{\partial}$-problem on $\H_{m/n}$ that is bounded from   $L^p_{(0,1)}(\H_{m/n}, |z_2|^{\alpha})$ to ${L^p(\H_{m/n}, |z_2|^{\alpha'})}$ for any $\alpha'>\alpha-  p \min\{1,\frac{n}{m}\}$.
\end{theorem}

\begin{proof}
	 We first construct the operator  and prove its boundedness on weighted $L^p$ spaces on $\H_{m/n}$. Then we will prove that it is indeed a solution to the $\bar{\partial}$-problem on $\H_{m/n}$.  
	 For convenience, in this proof, $F\in X$ means $\|F\|_X\lesssim 1$. Thus, our aim is to contruct a solution operator $T_{p,\alpha}$ so that for $f \in L_{(0,1)}^p(\H, |z_2|^\alpha)$, $T_{p,\alpha} f \in L^p(\H, |z_2|^{\alpha'})$.
	 
	 We will utilize the branched covering map  
	 \begin{align} \label{eq:psi_map}
	 	\psi: \D \times \D^* &\to \H_{m/n} \\
	 	w=(w_1, w_2) &\mapsto z=(z_1, z_2) = (w_1 w_2^n, w_2^m).\notag
	 \end{align}
	 The complex Jacobian determinant of $\psi$ is $|J_{\mathbb{C}}\psi(w)| = m w_2^{m+n-1}$, while its real Jacobian determinant is 
	 \begin{equation} \label{eq:jacobian}
	 	|J_{\mathbb{R}} \psi(w)| = m^2 |w_2|^{2(m+n-1)}.
	 \end{equation}
	 Thus, if $ J = 2m + 2n - 2$ then
	 \begin{equation} \label{eq:changevar}
	 	\|f\|^p_{L^p(\H_{m/n}, |z_2|^{\alpha})} = m^2 \|f\circ \psi\|^p_{L^p(\D \times \D^*, |w_2|^{m\alpha+J})} .
	 \end{equation}
	 Let $g = \psi^* f$ be the pull-back of $f$. Then $g=g_1d\bar{w}_1+g_2d\bar{w}_2$ where	
	 \begin{equation} \begin{cases} \label{pullback_components}
	 	g_1(w)=  \bw_2^n f_1(w_1  w_2^n,w_2^m)\\
	 	g_2(w)=n \bw_1 \bw_2^{n-1} f_1(w_1  w_2^n,w_2^m)+ m\bw_2^{m-1} f_2(w_1  w_2^n,w_2^m).
	 \end{cases}\end{equation}
Since $\psi$ is holomorphic, $g$ is $\bar{\partial}$-closed on $\D \times \D^*$, hence so is $w_2^l g$ for any $l \in \mathbb{Z}$. We will now choose a specific value $l$ in order to apply Theorem \ref{thm:canonical_bidisk}. The change of variable formula \eqref{eq:changevar} gives
$$\begin{cases}
	w_2^l g_1 \in L^p(D, |w_2|^{m\alpha + J - p(n+l)})\\
	w_2^l g_2 \in L^p(D, |w_2|^{m\alpha + J - p(\min\{m,n\}+l-1)}) .
\end{cases}$$
Note that we can slightly increase $\alpha$  if needed so that 
$$m\alpha+2(m+n) \notin p\Z$$
while preserving the condition $\alpha'>\alpha-  p \min\{1, n/m\}$.
Then there is a unique $l\in \Z$ satisfying
$$ \beta := m\alpha + J - p(\min\{m,n\}+l) \in(-2,p-2).$$

Since $w_2^l g_1\in L^p(D,|w_2|^\beta)$, by Lemma \ref{extension}, $w_2^l g$ extends to a $\bar{\partial}$-closed form on $D$. Thus, by Theorem \ref{thm:canonical_bidisk} 
\begin{equation}
	K_{D}(w_2^lg)\in L^p(D, |w_2|^{m\alpha + J - p(\min\{m, n\}+l)})
\end{equation}
solves $\dbar K_{D}(w_2^lg) = w_2^lg $ in $D$, hence, 
\begin{equation} \label{eq:K}
	K_{D,l} (g) := w_2^{-l}K_{D}(w_2^lg) \in L^p(D, |w_2|^{m\alpha + J - p\min\{m, n\}})
\end{equation}
solves $\dbar K_{D,l}(g) = g $ in $\D\times\D^*$.

When $m=1$, $\psi$ is a biholomorphism, so a solution operator on $\H_{1/n}$ is simply 
$T_{p,\alpha} f = K_{D,l} (g)\circ \psi^{-1}$. When $m > 1$, $\psi$ is an $m$-to-$1$ covering with a cyclic deck transformations group of order $m$ generated by
\[ \sigma(w_1, w_2) = (e^{-2\pi i n /m} w_1, e^{2\pi i / m} w_2), \]
so we take average over the fibers 
\begin{equation}\label{eq:V}
	v = \frac{1}{m} \sum_{j=0}^{m-1} K_{D,l} (g)\circ \sigma^j.
\end{equation}
Since  $\psi = \psi \circ \sigma^j$, $(\sigma^j)^* g = g$ for all $j$. It follows that $\bar{\partial} v = g$ in $\D\times\D^*$. 
As $v= v\circ \sigma$, 
\begin{equation} \label{eq:full_T_def}
	T_{p,\alpha} f (\psi (w)) := v (w)
\end{equation}
is well-defined. Furthermore, from \eqref{eq:changevar} and \eqref{eq:K}, we have
$$T_{p,\alpha} f \in L^p(\H, |z_2|^{\alpha - p\min\{1, \frac{n}{m}\}}),$$
which is the desire estimate.
	
We will now show that $\dbar T_{p,\alpha} f = f$. Let $\eta$ be a $(2,1)$-form with compact support on $\H_{m/n}$. Then using $\bar\partial v = g$ we have
\begin{align*}
	0 &=\int_{\D\times\D^*} \dbar(v\wedge \psi^*\eta) = \int_{\D\times\D^*} \dbar v\wedge \psi^*\eta + v\wedge \dbar\psi^*\eta  \\
	 & = \int_{\D\times\D^*} \psi^*f \wedge \psi^*\eta + \psi^* T_{p,\alpha} f \wedge \psi^*(\dbar\eta ) \\
	 & = m \int_{\H} f \wedge \eta + T_{p,\alpha} f \wedge \dbar\eta .
\end{align*}
This completes the proof.


\end{proof}

\section{Proof of Theorem \ref{thm:canonical_rational}}
\noindent{\it (i) Boundedness for $p\in(p_0,p_2)$:}

As noted in the Introduction, for $p \in (p_1, p_2)$, the result follows immediately from the known $L^p$ boundedness of $B_{\H_{m/n}}$ established in \cite{EdMc17, ChZe16} and the formula
\[K_{\H_{m/n}} f=(I-B_{\H_{m/n}})T_pf.\]

To extend the range down to  $p_0$, we utilize the following result for Bergman--Toeplitz operators:
\begin{theorem}[\cite{BTRWZ22}] \label{thm:toeplitz_boundedness}
	Let $0 \le \alpha \le \frac{m+n-1}{m}$ and $B_{\alpha}u:= B_{\H_{m/n}}\left(|z_2|^\alpha u \right)$. Then  $B_{\alpha}$ is bounded from $L^p(\H_{m/n})$ to $A^p(\H_{m/n})$  if and only if 
	\begin{equation} \label{eq:toeplitz_range}
		\frac{2m+2n}{m+n+1+m\alpha} < p < \frac{2m+2n}{m+n-1}.
	\end{equation}
\end{theorem}

We first note that if $p\in(p_0,p_2)$ then there exists  $0<\alpha<\min\{1,\frac{m}{n}\}$ such that 
\[\frac{2m+2n}{m+n+1+m\alpha} < p.\]
With this $\alpha$, using
\[ B_{\H_{m/n}} T_p f  = B_{\alpha}\left(|z_2|^{-\alpha} T_p f \right) \]
and  Theorems  \ref{thm:toeplitz_boundedness} and \ref{thm:main_weighted}, we obtain
\[\|B_{\H_{m/n}} T_p f  \|_{L^p(\H_{m/n})} \lesssim \||z_2|^{-\alpha} T_p f \|_{L^p(\H_{m/n})} \lesssim \|f \|_{L^p(\H_{m/n})}.\]

\noindent{\it (ii) Unboundedness for $p\ge p_2$}:

We will in fact prove the stronger fact that the canonical solution on $\H_{m/n}$ is not bounded from $L^\infty(\H_{m/n})$ to any Lorentz space $L^{p_2,q}(\H_{m/n})$ with $q<\infty$. Our example will use a simple extension of the well-known fact that $B_{\H_1}(\bar{z}_2) = cz_2^{-1}$.
	
	Since $\gcd(m,n)=1$, there exist $k\ge0$ and $l\ge 1$ so that  $$n(k+1) + m(-l+1) = 1.$$ Let $f =\bar{\partial}(z_1^k \bar{z}_2^l) \in L_{(0,1)}^\infty(\H_{m/n}) $. The canonical solution corresponding to $f$ is then
\[ K_{\H_{m/n}} f = z_1^k \bar{z}_2^l - B_{\H_{m/n}} (z_1^k \bar{z}_2^l)=z_1^k \bar{z}_2^l -c z_1^k z_2^{-l}. \]
A simple computation show that $z_1^k z_2^{-l} \notin L^{p_2,q}$ if $n(k+1) + m(-l+1) \leq 1$. Thus, $K_{\H_{m/n}} f\notin L^{p_2,q}$

\appendix
\section{Sharp Integral Estimates on the Unit Disk}

In this section, we establish the weighted $L^p$ estimates for the family of operators on $\D$ defined by 
\[ S_{a,b,c}f(z) = \int_{\D} s_{a,b,c}(z,w) f(w) dA(w), \]
where
\[ s_{a,b,c}(z,w) = |1-z\bar{w}|^a (1-|w|^2)^b |z-w|^c.\]
To do so, we will use Schur's lemma, together with sharp estimates for the integral
\begin{equation} \label{eq:I_general}
	I_{a,b,c,d}(z) = \int_{\D} |1-z\bar{w}|^a (1-|w|^2)^b |w-z|^c |w|^d dA(w),
\end{equation}
where $a, b, c, d \in \mathbb{R}$ and $z\in\D$.
To describe the behaviour of these integrals we use the following auxiliary function defined for $a \in \mathbb{R}$ and $t \in (0, 1)$
\begin{equation} \label{eq:Ta_def}
	T_a(t) = \begin{cases} 
		1 & \text{if } a > 0, \\ 
		1 - \log t & \text{if } a = 0, \\ 
		t^a & \text{if } a < 0. 
		\end{cases}
\end{equation}
Our estimate for $I_{a,b,c,d}(z)$ is as follows.
\begin{lemma} \label{lm:3_app}
	Suppose $b > -1$ and $c, d > -2$. Then for $z \in \D$, we have:
	\[ I_{a,b,c,d}(z) \lesssim T_{a+b+c+2}(1-|z|) T_{c+d+2}(|z|). \]
\end{lemma}

Note that this is a generalization of the estimate for $I_{-2,b,0,0}$ with $b\in(-1,0)$ of Forelli and Rudin  \cite{FoRu74}, for $I_{-2,b,0,d}$ with $b\in(-1,0), d\in(-2,0)$ of Edholm and McNeal \cite{EdMc16}, and for $I_{a,b,0,d}$ with $a\le -2$, $b\in(-1,0), d>-2$ of Khanh, Liu and Thuc \cite{KhLiTh17}. Our proof first  generalizes the result for this special case $c=0$. 
\begin{lemma}\label{lm:2}
	Suppose $b>-1$ and $d>-2$.
	Then 
	\begin{equation}
		I_{a,b,0,d}(z)\lesssim T_{a+b+2}(1-|z|). 
	\end{equation} 
\end{lemma}	
\begin{proof}
	
	We have 
	$$ I_{a,b,0,d}(z) = \int_{0}^1 \int_0^{2\pi}|1-r|z|e^{i\theta}|^a (1-r^2)^br^{d+1} d\theta dr.$$
	Let $l=r|z|$, then by the change of variable $s=1+\cos\theta$,
	\begin{align*}
		\int_0^{2\pi} |1-le^{i\theta}|^a d\theta &
	 = \int_{0}^{2} \frac{((1-l)^2+2ls)^{a/2}}{\sqrt{s(2-s)}}ds 
	 \lesssim \int_{0}^{2} \frac{\max\{(1-l)^2,s\}^{a/2}}{\sqrt{s(2-s)}}ds \\
	 &\lesssim \int_{0}^{(1-l)^2}  \frac{(1-l)^{a}}{\sqrt{s}}ds +  \int_{(1-l)^2}^{2} \frac{s^{a/2}}{\sqrt{s(2-s)}}ds 
	 \lesssim T_{a+1}(1-l).
	\end{align*}
	If $a\ge -1 $ then 
	$$I_{a,b,0,d}(z)  \lesssim \int_{0}^1 [1 -\ln(1-r|z|)] (1-r)^b r^{d+1} dr \lesssim 1 \lesssim T_{a+b+2}(1-|z|).$$	
	If $a< -1 $ then by 
	\begin{align*}
		I_{a,b,0,d}(z) &\lesssim \int_{0}^1  (1-r|z|)^{a+1} (1-r)^b r^{d+1} dr \\
		&\lesssim \int_{0}^{1/2}  r^{d+1} dr +  \int_{1/2}^1 (1-r|z|)^{a+1} (1-r)^b dr\\
		& \lesssim 1+ \int_{0}^{|z|} (1-r)^{a+b+1}  dr+ (1-|z|)^{a+1}\int_{|z|}^{1}   (1-r)^b  dr \\
		& \lesssim T_{a+b+2}(1-|z|) .
	\end{align*}
\end{proof}

\begin{proof}[Proof of Lemma \ref{lm:3_app}] We split the domain of integration into two sets, $\D=D_1\cup D_2$ where $D_1=\{w\in \D: |w-z|\le \frac{1-|z|}{2}\}$ and $D_2=\{w\in \D: |w-z|\ge \frac{1-|z|}{2}\}$.
 
 In $D_1$, we have  $1-|w| \sim 1-|z|$. Writing $1-z\w = 1-|z|^2+z(\w-\z)$ we see that $|1-z\w|
 \sim 1-|z|$.  Thus, the integral on this set is bounded by
\begin{align*}
	(1-|z|)^{a+b}\int_{D_1} |w-z|^c|w|^d \,dA(w).
\end{align*}
If $|z|>3/4$ then $|w|\sim 1 $ in $D_1$, hence 
$$(1-|z|)^{a+b}\int_{D_1} |w-z|^c|w|^d \,dA(w) \lesssim (1-|z|)^{a+b}\int_{D_1} |w-z|^c \,dA(w) \lesssim (1-|z|)^{a+b+c+2}.$$
If $|z|\le 3/4$ then
\begin{align*}
	& (1-|z|)^{a+b}\int_{D_1} |w-z|^c|w|^d \,dA(w)  \lesssim \int_{\D} |w-z|^c|w|^d \,dA(w)\\
	& \lesssim |z|^c \int_{\{|w|\le |z|/4\}} |w|^d \,dA(w) 
	\quad +|z|^d  \int_{\{|w-z|\le |z|/4\}} |w-z|^c \,dA(w)  \\
	&\quad + \int_{\{|w|,|w-z|\ge |z|/4\}} |w|^{c+d} \,dA(w) \\
	& \lesssim T_{c+d+2}(|z|).
\end{align*}

On $\partial D_2= \{ |w-z| = \frac{1-|z|}{2}\} \cup \partial \D$ we have  $|1-z\bar w| \sim |w-z|$ with implicit constants independent of $z$. Thus, by maximum principle, $|1-z\bar w| \sim |w-z|$ on $D_2$. Therefore,
$$\int_{D_2} |1-z\bar w|^a(1-|w|^2)^b |w-z|^c |w|^d \,dA(w)\lesssim I_{a+c,b,0,d}(z) \lesssim T_{a+b+c+2}(1-|z|), $$
by the previous lemma.
\end{proof}

The next lemma is a version of Schur's Lemma. We will use the notation
\[\|K(x,y)\|_{\Lnorm{y}{p}\Lnorm{x}{q}} = \|\|K(x,y)\|_{\Lnorm{x}{q}}\|_{\Lnorm{y}{p}}.\]

\begin{lemma}\label{schur}
	Let $T$ be defined by $Tf(x) = \int K(x,y)f(y) dy$. Suppose that $|K(x,y)| = K_1(x,y) K_2(x,y)$. Then for $1 \le p \le q \le \infty$, 
	\[ \|T\|_{\Lnorm{}{p} \to \Lnorm{}{q}} \le \|K_1\|_{\Lnorm{x}{\infty}\Lnorm{y}{p'}} \|K_2\|_{\Lnorm{y}{\infty}\Lnorm{x}{q}}. \]

\end{lemma}

\begin{proof}
	Let $f \in \Lnorm{}{p}$ and $g \in\Lnorm{}{q'}$. By Fubini's theorem, we have
	\begin{align*}
		|\langle Tf, g \rangle| &  = \left|\iint K_1(x,y) g(x) K_2(x,y) f(y)\, dy \, dx \right|\\
		& \le \left\| \| K_1(x,y) g(x) \|_{\Lnorm{y}{p'}}  \| K_2(x,y) f(y) \|_{\Lnorm{y}{p}} \right\|_{L^1_x} \\
		& \le \| K_1(x,y) g(x) \|_{\Lnorm{x}{q'} \Lnorm{y}{p'}} \| K_2(x,y) f(y) \|_{\Lnorm{x}{q} \Lnorm{y}{p}}\\
	    & \le \| K_1(x,y) g(x) \|_{\Lnorm{x}{q'}\Lnorm{y}{p'}} \| K_2(x,y) f(y) \|_{\Lnorm{y}{p} \Lnorm{x}{q}} \qquad\text{(by Minkowski)} \\
		& \le \| K_1(x,y) \|_{\Lnorm{x}{\infty} \Lnorm{y}{p'}} \|g\|_{\Lnorm{}{q'}} \| K_2(x,y) \|_{\Lnorm{y}{\infty} \Lnorm{x}{q}} \|f\|_{\Lnorm{}{p}}.
	\end{align*}
	Taking the supremum over all $f$ and $g$ with norm 1 yields the desired result.
\end{proof}

\begin{remark}
	The original version of Schur's Lemma is recovered by setting $p=q$ and choosing $K_1 = |K|^{1/p'}$ and $K_2 = |K|^{1/p}$. The more general version in \cite{KhLiTh17} is obtained by decomposing the kernel 
	$$K(x,y)\psi(y) = \frac{K^\eta h_1(y)}{g(x)} \frac{K^{1-\eta} g(x)\psi(y)}{h_1(y)}.$$
\end{remark}
In choosing the various parameters, we also need the following elementary lemma whose proof is omitted.
\begin{lemma} \label{element}
	There exists $x_1\in (\alpha_1,\beta_1)$, $x_2\in (\alpha_2,\beta_2)$ so that $x_2-x_1\in  [\alpha_3,\beta_3]$ if and only if these intervals are nonempty and $\alpha_1+\alpha_3<\beta_2$, $\alpha_2<\beta_1 +\beta_3$.
	
	In particular, nonempty intervals $(\alpha_1,\beta_1)$ and $(\alpha_2,\beta_2)$ intersect if and only if $\alpha_1<\beta_2$, $\alpha_2<\beta_1$.
\end{lemma}
We are now ready to prove the main result of this section.
\begin{theorem} \label{keybound} Assume that $1< p\le q< \infty$. Then $S_{a,b,c}: L^p(\D,|z|^\alpha) \to L^q(\D,|z|^\beta)$ is bounded if the following conditions hold
	\begin{itemize}
		\item $\alpha<2p-2$ and ${\beta}>-2$,
		\item $b>-\frac{1}{p'}$
		\item $c>-\frac{2}{p'}-\frac{2}{q}$, and $c\ge-\frac{2}{p'}-\frac{2}{q}+\frac{\alpha}{p}-\frac{\beta}{q}$
		\item $a+b+c\ge-\frac{2}{p'}-\frac{2}{q}$.
	\end{itemize}
	
\end{theorem}

\begin{proof}
	We first note that $S_{a,b,c}: L^p(\D,|z|^\alpha) \to L^q(\D,|z|^\beta)$ is bounded if and only if the operator with kernel
	$$K(z,w)  = s_{a,b,c}(z,w) |z|^{\beta/q} |w|^{-\alpha/p}$$
	is bounded from  $L^p(\D)$  to $L^q(\D)$.
	Let 
	$$K_1(z,w)=\frac{|1-z\bar w|^{a_1} (1-|w|^2)^{b_1}|w-z|^{c_1}|w|^{d_1}} {(1-|z|^2)^{b_2}|z|^{d_2-\beta/q}},$$
	and 
	$$K_2(z,w)=\frac{|1-z\bar w|^{a-a_1} (1-|z|^2)^{b_2}|w-z|^{c-c_1}|z|^{d_2}} {(1-|w|^2)^{b_1-b}|w|^{d_1+\alpha/p}}.$$
	Then $K=K_1 K_2$, so by Lemma \ref{schur}, it suffices to find  constants $a_1, b_1, b_2, c_1, d_1, d_2$ such that $$\|K_1\|_{L^{\smash{\infty}}_z L^{\smash{p'}}_w} + \|K_2\|_{L^{\smash{\infty}}_w L^{\smash{q}}_z } <\infty.$$
	
	By Lemma \ref{lm:3_app}, if $b_1>-1/p'$, and $c_1,d_1>-2/p'$ then
	$$\|K_1\|_{L^\infty_z L^{p'}_w} \lesssim \sup_{z\in\D} \frac{ T_{(a_1+b_1+c_1)p'+2}^{1/p'}(1-|z|)\cdot T_{(c_1+d_1) p'+2}^{1/p'}(|z|)  }{(1-|z|^2)^{b_2}|z|^{d_2-\beta/q}}.$$
	The right-hand side is finite if 
	\begin{equation}\label{cond1}
		\begin{cases}
			b_2 <0 \text{ and } b_2 \le  a_1+b_1+c_1 +\frac{2}{p'}\\
			d_2-\frac{\beta}{q}<0 \text{ and } d_2- \frac{\beta}{q} \le c_1+d_1+\frac{2}{p'} 
		\end{cases}
	\end{equation}
	
	Similarly, if  $b_2>-1/q$, and $c-c_1,d_2>-2/q$ then
	$$\|K_2\|_{L^\infty_w L^{q}_z} \lesssim \sup_{w\in\D} \frac{ T_{(a-a_1+b_2+c-c_1)q+2}^{1/q}(1-|w|)\cdot T_{(c-c_1+d_2)q +2}^{1/q}(|w|) }{(1-|w|^2)^{b_1-b}|w|^{d_1+\alpha/p}},$$
	which is finite if
	\begin{equation}\label{cond2}
		\begin{cases}
			b_1 - b <0 \text{ and }  b_1 -b \le  a-a_1+b_2+c-c_1 +\frac{2}{q}\\
			 d_1 + \frac{\alpha}{p}<0 \text{ and } d_1+ \frac{\alpha}{p} \le c-c_1+d_2+\frac{2}{q} 
		\end{cases}
	\end{equation}		
	Combining \eqref{cond1} and \eqref{cond2}, we need to find  $a_1\in \R$,  $c_1 \in (-\frac{2}{p'},c+\frac{2}{q})$ so that there exists
	\begin{enumerate}
		\item $d_1\in (-\frac{2}{p'},-\frac{\alpha}{p})$,  $d_2\in (-\frac{2}{q},\frac{\beta}{q})$ with $d_2 -d_1 \in [-c+c_1-\frac{2}{q}+\frac{\alpha}{p}, c_1 + \frac{\beta}{q}  +\frac{2}{p'} ]$. \label{aaa}
		\item $b_1\in(-\frac{1}{p'},b)$, $b_2 \in (-\frac{1}{q},0)$ with $b_2 -b_1 \in [-a-b-c+a_1+c_1 -\frac{2}{q}, a_1+c_1+\frac{2}{p'}]$.
	\end{enumerate}
	Note that by our hypothesis, all intervals above are non-empty. Thus, be Lemma \ref{element}, there exist $d_1,d_2$ satisfying (1) if and only if
	       	$$c_1\in (-\frac{2}{p'}-\frac{2}{q}+\frac{\alpha}{p}- \frac{\beta}{q} , c+\frac{2}{p'}+\frac{2}{q}-\frac{\alpha}{p} +\frac{\beta}{q}) .$$
	By our assumptions, this interval is nonempty and intersects $(-\frac{2}{p'},c+\frac{2}{q})$, by Lemma \ref{element}. Thus, $c_1$ can be chosen and thence $d_1$ and $d_2$.
	
	By  Lemma \ref{element} again, there exists $b_1,b_2$ satisfying (2) if and only if
		$$a_1\in (-c_1-b- \frac{1}{q} -\frac{2}{p'}, a+b+c-c_1+\frac{2}{q}+\frac{1}{p'}).$$
	The interval is non empty since by our assumptions, 
	$$a+2b+c > -\frac{3}{p'} -\frac{3}{q}.$$
	Thus, $a_1$ can be chosen and thence $b_1$ and $b_2$. The proof is complete.
\end{proof}
	\bibliographystyle{alpha}
\bibliography{mybib-Khanh-Tu}
\end{document}